\documentclass[12pt,a4paper,oneside]{amsart}

\usepackage{amsfonts, amsmath, amssymb, amsthm, amscd}
\usepackage{anysize}
\usepackage{mathtools}
\marginsize{2cm}{2cm}{1.5cm}{1.5cm}

\usepackage[T2A]{fontenc}
\usepackage[utf8]{inputenc}
\usepackage[english]{babel}
\usepackage{cmap}
\usepackage{enumerate}

\usepackage{xcolor}
\usepackage{hyperref}
\hypersetup{
	colorlinks   = true, %Colours links instead of ugly boxes
	urlcolor     = blue, %Colour for external hyperlinks
	linkcolor    = blue, %Colour of internal links
	citecolor    = red %Colour of citations
}
\newcommand{\Href}[2]{\hyperref[#2]{#1~\ref{#2}}}
\usepackage{anysize}
\marginsize{2cm}{2cm}{1.5cm}{1.5cm}

\newtheorem{thm}{Theorem}
\newtheorem{prp}{Proposition}[section]
\newtheorem{lem}{Lemma}[section]
\newtheorem{cor}{Corollary}[section]

\newtheorem{claim}{Claim}[section]

\theoremstyle{definition}
\newtheorem{dfn}{Definition}[section]

\newcommand{\st}{:\;}

\newcommand{\norm}[1]{\left\|#1\right\|}
\newcommand{\enorm}[1]{\left|#1\right|}

\newcommand{\conv}{\mathrm{conv}}%

\newcommand{\iprod}[2]{\left\langle#1,#2\right\rangle}%

\def\N{{\mathbb N}}
\def\R{{\mathbb R}}

\def\Lin{\mathop{\rm Lin}}

\def\phi{\varphi}
\def\epsilon{\varepsilon}

\newcommand{\dist}[2]{\operatorname{dist}\!\left( #1, #2 \right)}%

\newcommand{\la}{\lambda}

\newcommand{\e}{\varepsilon}

\newcommand{\BF}[2]{\operatorname{B}\nolimits_{#1}\!\left(#2\right)}%

\DeclareMathOperator{\mglname}{\varrho}

\newcommand{\mglx}[1]{\mglname\nolimits_{X}\!\!\!\:\left( #1 \right)}%

\DeclareMathOperator{\mconame}{\delta}
\newcommand{\mcox}[1]{\mconame\nolimits_{X}\!\left( #1 \right)}%

\newcommand{\zetap}[1]{\zeta^{+}_{X}\!\left( #1 \right)}%

\newcommand{\omx}[1]{\omega_{X}\!\left( #1 \right)}%
\newcommand{\omxi}[1]{\omega_{X}^{-1}\!\left( #1 \right)}%

\newcommand{\psif}[1]{\psi\!\left( #1 \right)}%

\newcommand{\length}[1]{\mathrm{length}\left(#1\right)}

\usepackage[ruled,vlined]{algorithm2e}

\title{Rectifiable curves in  proximally smooth sets}
\usepackage[inline]{showlabels}
\renewcommand{\showlabelsetlabel}[1]%
{\showlabelfont \href{#1}{LABEL}}

\author{Grigory Ivanov\address{Grigory Ivanov: 
Institute of Science and Technology Austria (IST Austria), 
Kleusteneuburg, 3400, Austria; Laboratory of Combinatorial and Geometrical Structures, Moscow Institute of Physics and Technology, Moscow, 141701, Russia}
\email{grimivanov@gmail.com}
\and
Mariana Lopushanski \address{Steklov Mathematical Institute of the Russian Academy of Sciences, 
Moscow, Russia}
\email{masha.alexandra@gmail.com }
}

\thanks{ Supported by Russian Science Foundation, project N 19-11-00087}

\date{\today}

\begin{document}
\begin{abstract}
We provide  an algorithm of constructing a rectifiable curve between two sufficiently close points  of a proximally smooth
set in a uniformly convex and uniformly smooth Banach space. 
Our algorithm returns a reasonably short curve between two sufficiently close points of a proximally smooth set, is iterative and uses a certain modification of the metric projection. We estimate  the length of a constructed curve and its deviation from the segment with the same endpoints. These estimates coincide up to a constant factor with those for the geodesics in a proximally smooth set  in a Hilbert space.
\end{abstract}
\maketitle

\section{Introduction}

Weakly convex sets have been studied in non-smooth analysis for several decades. Several established mathematicians
proposed their own definition of a weakly convex set,
among them are Federer \cite{federer1959curvature},
 Efimov and Stechkin \cite{efimov1958some}, Vial \cite{vial}, 
Rockafellar (for akin classes of functions) \cite{rockafellar1981favorable}.
In this paper we stick to the most convenient in our opinion definition of a weakly convex set  due to Clarke, Stern and Wolenski \cite{Clarke1995}. A closed set in a Banach space is called
\emph{proximally smooth with constant $R$} if
 distance to it from a point of the space is continuously
differentiable in the open $R$-neighborhood of this set excluding the set itself (this and other definitions used in the introduction  are formally given below in  \Href{Section}{sec:notation}).  
As for other definitions, each of them characterizes  weakly convex sets as the set with a certain property, e.g. differentiability of the distance function, supporting by balls, hypomonotonicity of the normal cone, etc.
It turns out that many of these definitions are equivalent in Hilbert space, which allows using weakly convex sets in different applications. See, for example, \cite{Balashov}  and \cite{IvanovGames}.
However, everything is a bit trickier in Banach spaces. 
Different classes of weakly convex sets in Banach spaces were studied in \cite{Bernard2006} , \cite{Bernard2011}, \cite{Alimov}. As was shown in \cite{IvBal2009}, some of the definitions are still equivalent
in uniformly smooth and uniformly convex Banach spaces, but some not. The first author \cite{Ivanov:233858} showed that the hypomonotonicity of the normal cone of a closed set  fails to be  equivalent to the proximally smoothness in any Banach space that is not isomorphic to a Hilbert space (in a Hilbert space these two properties are equivalent, see \cite[Theorem 1.9.1]{IvMonEng} and \cite[Corollary 2.2]{Pol_Rock_Thi1}).

Another equivalent to  the proximal smoothness property of a closed set in a Hilbert space  was given in \cite[Theorem 1.14.2]{IvMonEng}. We formulate it as follows.
\begin{prp}
Let $A$ be a closed set in a Hilbert space and $R > 0.$ The
following conditions are equivalent:
\begin{enumerate}
\item The set  $A$ is proximally smooth with constant $R.$ 
\item  For any two different points $x_0, x_1 \in A$ with 
$\enorm{x_0 - x_1} < 2R,$ there exists a curve $\gamma$ in $A$ with  endpoints $x_0$ and $x_1,$ whose length is at most
\[
2R \arcsin \left(\frac{\enorm{x_0 - x_1}}{2R} \right).
\]
\end{enumerate}
\end{prp}
\noindent This result plays a crucial role in proofs of many other important results related to the properties and applications of proximally smooth sets. For example, it implies the existence and uniqueness of the  shortest path connecting two sufficiently close points of a proximally smooth set in a Hilbert space, and that a ``locally'' proximally smooth set in a Hilbert space is proximally smooth.  

However,  even in a sufficiently smooth and convex 
Banach space, the existence of rectifiable curve between to sufficiently close points of a proximally smooth set has been unknown yet. One might argue that the definition of a proximally smooth set implies locally connectedness of a proximally smooth set
since the metric projection of a sufficiently short segment with endpoints in a proximally  smooth set onto the set itself has to be continuous (to be more precise, this argument works in a uniformly smooth and uniformly smooth spaces for segments strictly shorter than $2R$). It is not clear whether  the curve
constructed in such a way is rectifiable since the metric projection onto a proximally smooth set is H\"older continuous \cite[Theorem 3.2]{ivanov2015sharp}, but not Lipschitz continuous in a Banach space not isomorphic to a  Hilbert space. Even if one can show that such a curve is rectifiable, the same rather unfortunate property of the metric projection implies that this curve has quite nasty behavior.

In this paper we provide  an algorithm of constructing a rectifiable curve between two sufficiently close points  of a proximally smooth
set in a uniformly convex and uniformly smooth Banach space. 
Our \Href{Algorithm}{algorithm} returns a reasonably short curve between two sufficiently close points of a proximally smooth sets, is iterative and uses a certain modification of the metric projection.  
We collect two important properties of the curve constructed with the use of our algorithm in the two following Theorems.
\begin{thm}\label{thm:rectifable_curve} 
Let $X$ be a uniformly convex and uniformly smooth  Banach space whose modulus of smoothness is of power type $s.$
  Then there are positive constants $\beta_L$ and $L$ satisfying the following property. 
Let  $A \subset X$ be  a proximally smooth set with positive constant  $R,$  and let  $x_0, x_1 \in A$ with   
$\frac{\norm{x_0 -x_1}}{R} < \beta_L.$
Then \Href{Algorithm}{algorithm} returns the curve $\gamma$  in $A$ with endpoints $x_0$ and $x_1$
such that   inequality  
\begin{equation}
\label{eq:thm_length_of_curve}
%\ell ( \gamma ) \leq \norm{x_0 - x_1}
%\exp\left[ \beta_2 \left(\frac{\norm{x_0 - x_1}}{R}\right)^{s(s-1)}\right] 
\length{ \gamma } \leq \norm{x_0 - x_1}
\left( 1 + L \left(\frac{\norm{x_0 - x_1}}{R}\right)^{s(s-1)}\right)
\end{equation}
holds.
\end{thm}
\begin{thm}\label{thm:rectifable_curve_inclusion} 
Let $X$ be a uniformly convex and uniformly smooth  Banach space whose modulus of smoothness is of power type $s.$
  Then there are positive constants $\beta_I$ and $L_I$ satisfying the following property. 
Let  $A \subset X$ be  a proximally smooth  with positive constant  $R$
set,  and let  $x_0, x_1 \in A$ with   
$\frac{\norm{x_0 -x_1}}{R} < \beta_I.$
Then \Href{Algorithm}{algorithm}  returns the curve $\gamma$  in $A$ with endpoints $x_0$ and $x_1$
such that   inclusion 
\begin{equation}
\label{eq:thm_inclusion_curve}
 \gamma \subset  
 \conv \left\{
 x_0, \BF{r}{\frac{x_0 + x_1}{2}}, x_1 \right\}
\end{equation}
holds, where $\BF{r}{\frac{x_0 + x_1}{2}}$ is the closed ball centered at $\frac{x_0 + x_1}{2}$ of radius 
\[
r = L_I \norm{x_0 - x_1 } 
\left(\frac{\norm{x_0 - x_1}}{R}\right)^{s-1}. 
\]
\end{thm}
We note that we will estimate constants $\beta_L, L, \beta_I, L_I$ using  constants related to the smoothness of Banach space $X.$ Moreover, we do not use the uniform convexity of $X$ directly; the reason for using this condition on a space is being able to use different definitions of a weakly convex set, which are equivalent to the proximally smoothness in a 
uniformly convex and uniformly smooth Banach space. This condition might be relaxed, for example, in a finite dimensional space, but it will add complications and little to the ideas. Also, since the complement of the interior of the unit ball is proximally smooth with constant one and by basic properties of the modulus of smoothness, the bound on $r$ in  \Href{Theorem}{thm:rectifable_curve_inclusion} is asymptotically tight.
The bound on the length of the constructed curve in  
\Href{Theorem}{thm:rectifable_curve} coincides up to a constant factor with the bound on the shortest path between two points on a fixed distance in a proximally smooth set in uniformly smooth space whose modulus of smoothness is of power type 2 (for example, in a Hilbert space and in $L_p$ spaces with $p \geq 2$).

Since the metric projection onto a proximally smooth set in a uniformly convex and uniformly smooth Banach space is H\"older continuous \cite[Theorem 3.2]{ivanov2015sharp}, we get the following result as an immediate consequence of \Href{Theorem}{thm:rectifable_curve}.
\begin{cor}
Let $X$ be a uniformly convex and uniformly smooth  Banach space whose moduli of smoothness and convexity are of power type.
Let  $A \subset X$ be  a proximally smooth set with positive constant  $R,$  and let  $x_0, x_1 \in A$ with   
$\frac{\norm{x_0 -x_1}}{R} < 2.$
Then there is a rectifiable curve $\gamma$ in $A$ with endpoints $x_0$ and $x_1.$
\end{cor}

The rest of the paper is organized as follows.
In the next \Href{Section}{sec:notation}  we give the standard terminology related to weakly convex sets  and to the geometry of the unit ball of a Banach space. In  \Href{Section}{sec:distance_func} we study the distance function  to  a proximally smooth set  restricted to  a segment with endpoints in this set. 
In \Href{Section}{sec:curve_construction}
we describe our \Href{Algorithm}{algorithm}  for the construction of the rectifiable curve with endpoints in a proximally smooth set, and also summarize the assumptions needed to show the correctness of the algorithm. Then in \Href{Section}{sec:length_bound} we estimate the length of  curve returned by \Href{Algorithm}{algorithm} and prove
 \Href{Theorem}{thm:rectifable_curve}. In \Href{Section}{sec:inclusion} we prove \Href{Theorem}{thm:rectifable_curve_inclusion}. Finally, in the last section \Href{Section}{sec:tech_lemmas} we prove several purely technical results used in the proofs.

\section{Terminology and basic properties}
\label{sec:notation}
\subsection{Properties of the unit ball}

Let $X$ be a {\it real Banach space}, and $X^*$ be its {\it conjugate space}. 
We use $\iprod{p}{x}$ to denote the \emph{value of a functional} $p \in X^*$ \emph{at a vector} $x \in X.$
For $r > 0$ and $c \in X$ we denote by $\BF{r}{c}$ the \emph{closed ball} 
{\it with center $c$ and radius $r.$}
% and
%by $\B_R^{*}(c)$ the respective \emph{ball in the conjugate space}. Thus, $ \US$ denotes the \emph{unit sphere} of $X.$

%By definition, we put $J_1(x) = \{p \in \USD :\, \iprod{p}{x} = \norm{x}\}.$

We will use the notation $[xy]$ for the {\it segment} with  endpoints $x$ and $y.$

Define
\begin{equation*}
\mcox{\e} = \inf \left\{ 1 - \frac{\norm{x + y}}{2} \st x,y\in \operatorname{B}_1 (0),\ \norm{x -y} \geq \e\right\}
\end{equation*}
and
\begin{equation*}
\mglx{\tau} = \sup \left\{\frac{\norm{x + y} + \norm{x - y}}{2} - 1 \st \norm{x} = 1, \norm{y} = \tau \right\}.
\end{equation*}
Functions $\mcox{\cdot} \colon [0,2] \to [0, 1]$ and $\mglx{\cdot} \colon \R^+ \to \R^+$
 are  referred to as the \emph{moduli of convexity and smoothness of} $X$, respectively.
The modulus of convexity is of \emph{power type} $s > 0$ if for some constant $C_{cv}$ inequality
\begin{equation}
\label{eq:modulus_smoothness_power_type}
\mcox{\epsilon} \geq C_{cv} \tau^s
\end{equation}
holds for any $\tau \in [0, 2).$

A Banach space $X$ is called \emph{uniformly convex} if $\mcox{
\e}>0$ for all $\e>0$, and \emph{uniformly smooth} if $\frac{\mglx{\tau}}{\tau}\to 0$ as $\tau\to 0.$
We refer the reader to the book \cite{DiestelEng}  as a comprehensive survey on these moduli and their geometric properties.

In what follows, we consider  only uniformly smooth Banach spaces. In such spaces, for any non-zero vector $x,$ 
there is a unique unit functional $p$  attaining its norm on
 $x.$  
Let $x$ be a non-zero vector of $X$ and $p$ be a unit functional attaining its norm on $x,$ we use
$H_{x}$ to denote the hyperplane $\{y \in X \st \iprod{y}{p}=0\}.$
We will say that $y$ is {\it quasi-orthogonal}  to vector $x \in X\setminus \{0\}$
and write  $y\urcorner x$
if $y \in H_x.$

Note that  the following conditions are equivalent: \\ \noindent
 -- $y$  is quasi-orthogonal to  $x$; \\ \noindent
 --  for any $\la \in \R$  vector $x+ \la y$
 lies in the supporting hyperplane to the  ball  $\BF{\norm{x}}{0}$  at $x;$ \\ \noindent
 -- for any $\la \in \R$ inequality $\norm{x + \la y} \geq \norm{x}$ holds;\\ \noindent
 --  $x$ is orthogonal to  $y$ in the sense of Birkhoff--James (see \cite[Chapter 2]{DiestelEng} and \cite{AlonsoMartiniWu_birkhoff_orthogonality}).

\subsection{Modulus of smoothness and related functions}
The modulus of smoothness of a Banach space is a strictly increasing convex function satisfying the following inequality
of Day--Nordlander type (see \cite[Chapter 3]{DiestelEng})
\begin{equation}
\label{eq:day_nord_mglx}
\sqrt{1 + \tau^2} - 1 \leq \mglx{\tau} \leq \tau
\quad \text{for all} \ \tau \in [0, + \infty).
\end{equation}

The modulus of smoothness is of \emph{power type} $s$ if for some constant $C_{sm},$  
\begin{equation}
\label{eq:modulus_smoothness_power_type}
\mglx{\tau} \leq C_{sm} \tau^s
\quad \text{for all } \ \tau \in [0, + \infty).
\end{equation}
It follows that the modulus of smoothness of a uniformly smooth Banach space  might be of power type $s$ only for some $s$
in $(1, 2].$

In our computations we will use two functions related to the modulus of smoothness of a Banach space. 

Define function $\omega_X:  [0,+\infty) \to  [0,+\infty)$ by
$$
	\omx{\tau} = \frac{\mglx{\tau}}{\tau}.
$$
Since the modulus of smoothness of a uniformly smooth Banach space is a strictly increasing convex function, we conclude that  $\omega_X(\cdot)$ is a strictly increasing function. 
Thus, the inverse function $\omxi{\cdot}$ is also strictly increasing.

The second function $\zeta_{X}^{+}: [0, + \infty)  \to [0, + \infty)$ is defined by 
\[
\zetap{\e} =  \sup\left\{\norm{x + \e y} : \,\, 
\norm{x} = \norm{y} = 1,  \, 
y \urcorner x \right\}.
\]
Thus, $\zetap{\cdot} -1$ bounds  the   \emph{deviation of a point in a supporting hyperplane from the unit ball}.
This modulus of a Banach space was studied in \cite{ivanov2017new}, where  it was shown that it is equivalent to the modulus of smoothness near zero.
\begin{prp}
\label{prop:zetap_mglx_equivalence}
	Let  $X$ be an arbitrary Banach space. Then
\[
 \mglx{\frac{\e}{2(1+\e)}} \leq \zetap{\e}  -1 \leq \mglx{2\e}, \quad  \e \in\left[0, \frac{1}{2}\right].
 \]
\end{prp}
It is not hard to see that $\zeta_X^{+}$ is strictly increasing, and hence, its inverse function $\left(\zeta_X^{+}\right)^{-1}$ is well-defined and is strictly increasing.

\subsection{Weakly convex sets}

The distance from a point $x \in X$ to a set $A \subset X$ is defined as 
$$
\dist{x}{A} = \inf\limits_{a \in A} \norm{x - a}.
$$
The {\it metric projection} of a point $x $ onto a set $A$ is defined as any element of the set 
$$
	P_A(x) = \{a \in A \st \norm{a-x} = \dist{x}{A}\}.
$$
We call the set 
$
\left\{ x \in X \st 0 < \dist{x}{A} < R \right\}
$ 
the \emph{open $R$-neighborhood} of a set $A.$
\begin{dfn}
A  set $A \subset X$ is called proximally smooth 	with constant $R$ if it is closed and  the distance function
 $x \mapsto \dist{x}{A}$ is continuously differentiable on 
 the open $R$-neighborhood of $A.$ 
\end{dfn}
%We denote by $\Omega_{PS}(R)$ the set of all closed proximally smooth  sets
%with constant $R$ in $X.$ 

The geometric properties of proximally smooth sets are hidden in the definition. To clarify these geometrical properties, which are very useful in this paper,
we introduce two equivalent (in certain spaces) to the proximal smoothness properties.

\begin{prp}[\cite{IvBal2009}]
\label{prop:projection_continuity}
Let $X$ be a uniformly convex and uniformly smooth Banach space, let $A\subset X$ be a closed set,  and let $R>0.$ 
The following assertions are equivalent:
\begin{enumerate}
 \item $A$ is proximally smooth with constant $R.$ 
 \item \label{ass:equiv_pr_supp}  the projection map $x \to P_A(x)$ is single valued and continuous on the open $R$-neighborhood of $A.$
 \item\label{ass:equiv_ps_supp}  for any $u$ in the open $R$-neighborhood of $A$  and any $x \in P_A(u)$  	 inequality
 \[
 \dist{x + \frac{R}{\norm{u-x}}(u-x)}{A} \geq R
 \]
 holds.  
\end{enumerate}
\end{prp}
 Roughly speaking, the last property here implies that the set can be supported by a ball of fixed radius $R$ at a point of its boundary.

\subsection{Auxiliary geometric constructions}
\label{subsec:construction}
\begin{dfn}
Let $A$ be  proximally smooth with constant $R,$
let $x_0$ and $x_1$  be two distinct points of $A$ with $\frac{\norm{x_0 - x_1}}{R} < 2,$ 
  we say that an arbitrary point of the set 
  \[ 
  P_{A}([x_0, x_1]) \cap \left(H_{x_0 - x_1} + \frac{x_0 + x_1}{2}\right)
  \] is  a \textit{slice-projection}   of the midpoint  $\frac{x_0 + x_1}{2}$ of the segment $[x_0 x_1]$ onto $A.$ 
  \end{dfn}
 \Href{Proposition}{prop:projection_continuity} and the separation lemma imply that the slice-projection is non-empty in a uniformly smooth and uniformly convex Banach space.

Given a point $x \in X$ and a set $D \subset X,$ we denote the  cone
\[
x + \left\{ 
\sum\limits_1^n \lambda_i (s_i - x) \st 
n \in \N; \lambda_i \geq 0 \  \text{and }\ s_i \in D  
 \ \text{for all} \ i \in [n]
\right\}
\]
as
$ \mathrm{cone}(x, D).$ 
Note that $ \mathrm{cone}(x, D)$ is a convex cone for any nonempty set $D.$  

In our computations we will extensively use the following quantity, which describes in a certain way the distortion of the distance function.
By definition put
\begin{equation}\label{eq:def_Rprime}
 R^\prime(\tau, R) = R \frac{8   \mglx{\frac{{\tau}}{R}}}
{1 - 8   \omx{\frac{\tau}{R}}} =
\tau \frac{8   \omx{\frac{{\tau}}{R}}}
{1 - 8   \omx{\frac{\tau}{R}}}.
\end{equation}

\section{Distance to a proximally smooth set}
\label{sec:distance_func}
In this section we bound the distance between the midpoint of a segment with endpoints in a proximally smooth set $A$ and its slice-projection onto $A.$

\begin{lem}\label{lem:dist_to_pr_sm_set}
Let $X$ be a uniformly convex and uniformly smooth Banach space, let $R>0$ and  $A \subset X$ be a proximally smooth with constant $R$ set. Let $x_0, x_1 \in A$ with $\frac{\norm{x_0-x_1}}{R} < 2$ and $ \la
\in [0,1]$. Then the following bound on the distance from point $x_{\la} = (1 - \la)x_0 + \la x_1$ to  set $A$ holds
\begin{equation}
 \label{eq:dist_to_pr_sm_set1}
 \dist{x_\la}{A} \leq 8 R \la (1 - \la) \mglx{\frac{\norm{x_0 - x_1}}{R}}.
 \end{equation} 
\end{lem}

\begin{proof}
\Href{Proposition}{prop:projection_continuity} implies that $P_A(x_\la)$ is nonempty for all $\la \in [0,1].$ Fix an arbitrary $\la\in (0,1)$ and consider $y\in P_A(x_\la)$.
Using assertion \ref{ass:equiv_ps_supp} of \Href{Proposition}{prop:projection_continuity}, we get that $\dist{y+R\frac{x_\la-y}{\norm{x_\la-y}}}{A}\ge R$. 
Hence $\norm{y+R\frac{x_\la-y}{\norm{x_\la-y}}-x_0}\ge R$,
and   $\norm{\frac{x_\la-y}{\norm{x_\la-y}}-\frac{y-x_0}{R}}\ge 1$. By the definition of the modulus of smoothness, we get that 
$$2\mglx{\norm{\frac{y-x_0}{R}}}\ge \norm{\frac{x_\la-y}{\norm{x_\la-y}}+\frac{y-x_0}{R}}+ \norm{\frac{x_\la-y}{\norm{x_\la-y}}-\frac{y-x_0}{R}}-2\ge$$
$$\norm{\frac{x_\la-y}{\norm{x_\la-y}}+\frac{y-x_0}{R}}-1.$$

Let $p$ be the unit functional attaining its norm on vector 
$x_{\la}-y.$ Then

$$\norm{\frac{x_\la-y}{\norm{x_\la-y}}+\frac{y-x_0}{R}}-1\ge\iprod{p}{\frac{x_\la-y}{\norm{x_\la-y}}+\frac{y-x_0}{R}}-1=\frac{1}{R}\iprod{p}{y-x_0}.$$

Therefore, we obtain that
\begin{equation} \label{eq:dist_y_eps}
\iprod{p}{x_0 - y} \leq 2R \mglx{
\frac{\norm{x_0 - y}}{R} }.
\end{equation}

Since $y\in P_{A}(x_\la)$ and $x_0, x_1 \in A$, 
\[
\norm{y-x_{\la}}=\dist{x_{\la}}{A} \leq 
\min\{\norm{x_{\la}-x_{0}},\norm{x_{\la}-x_{1}} \} \leq
\]
\[
\min\{\norm{x_{\la}-x_{0}},\norm{x_{\la}-x_{1}}\} \leq 
\min{\{\la, (1 - \la)\}} \norm{x_0 - x_1}.
\]
Therefore,
\[
\norm{y-x_0}  \leq  
\norm{y- x_{\la}}  + \norm{x_{\la} - x_0 }  
\leq \min\{\!2\la, 1\!\} \norm{x_0  - x_1}.
\]
This, the monotonicity of the modulus of smoothness and inequality  \eqref{eq:dist_y_eps} yield  inequality
\begin{equation}\label{eq:Phi_e_lambda}
\iprod{p}{x_0 - y} \leq \Phi(\la),
\end{equation}
 where
\begin{equation}\label{eq:Phi_of_mglx}
\Phi(\la)=2R \mglx{\frac{\min{\{2\la, 1\}}
\norm{x_0 - x_1}}{R}}.
\end{equation}
Similarly, 
\begin{equation}\label{eq:iprod_less_Phi}
\iprod{p}{x_1 - y} \leq \Phi(1-\la).
\end{equation}
Multiplying inequalities  \eqref{eq:Phi_e_lambda} and \eqref{eq:iprod_less_Phi} ­  by $(1-\la)$ and ­ 
$\la$ respectively, and then summing them, one has
$$
\iprod{p}{x_\la - y} \leq (1-\la)\Phi(\la) +
\la\Phi(1-\la).
$$
This and the inequality
$\dist{x_{\la}}{A} \leq \norm{x_{\la}-y} =
\iprod{p}{x_\la - y}$ imply that
\begin{equation}\label{eq:dist_less_Phi}
\dist{x_{\la}}{A} \le (1-\la)\Phi(\la) + \la\Phi(1-\la).
\end{equation}

Using once more the convexity of function $\mglx{\cdot}$,   and the identities 
$\mglx{0}=0$ and \eqref{eq:Phi_of_mglx}, we get that
\[
\Phi(\la)\le 2R\min\{2\la, 1\} \mglx{\frac{\norm{x_0 -x_1}}{R}} \leq 
4\la R \mglx{\frac{\norm{x_0 - x_1}}{R}}.
\]
Thus, inequality \eqref{eq:dist_less_Phi} implies that
$$
\dist{x_{\la}}{A} \leq 8R \la (1-\la) \mglx{\frac{\norm{x_0 - x_1}}{R}}.
$$ 
\end{proof}

\Href{Proposition}{prop:projection_continuity} and the separation lemma imply  the following.

\begin{lem}\label{lem:slice_proj_existence}
Let $X$ be a uniformly convex and uniformly smooth  Banach space, let $A \subset X$ be a proximally smooth with constant $R$ set. 
 Then  for any    $x_0, x_1 \in A$ with   $ \frac{\norm{x_0 - x_1}}{R} <  2,$ there exists a slice-projection of the midpoint of $[x_0x_1]$ onto $A.$
\end{lem}

%\begin{dfn}\label{def:weakly_convex_sausage_waist}
%Let $x_0$ and $x_1$ be two different points of a Banach space $X$.
%We define double-cone segment of radius $r$, denoted  $D_r(x_0,x_1),$ as
%\[
%\conv\left\{x_0, \; \BF{r}{\frac{x_0 + x_1}{2}} \cap 
%\HP{x_1 - x_0}{\frac{x_0 + x_1}{2}},\; x_1 \right\}.
%\]
%\end{dfn}

\begin{lem}\label{lem:weakly_convex_sausage_waist}
Let $X$ be a uniformly convex and uniformly smooth  Banach space, let   $A \subset X$ be a proximally smooth with constant $R$ set.
Let $x_0, x_1 \in A$ with   $\frac{\norm{x_0 - x_1}}{R} <   \omxi{1/8},$ the following inclusion holds
\[
P_A([x_0x_1]) \subset 
\conv\left\{x_0, \; \BF{r^{\prime}}{\frac{x_0 + x_1}{2}} \cap 
\left(H_{x_1 - x_0} + \frac{x_0 + x_1}{2}\right),\; x_1 \right\}, 
\]
where
$ 
r^\prime = R^\prime(\norm{x_0 - x_1}, R)
$  (see \eqref{eq:def_Rprime}).

\end{lem}
\begin{proof}
We fix a unit vector $y$ quasi-orthogonal to $x_1 - x_0$ 
%(that is, 
%$\iprod{p}{y}=0$) 
and consider the two-dimensional space $X_2 = \Lin\{y, x_1 - x_0\}$ with the induced norm. Fix $\la \in (0,1)$ and set  $r =  8R \la (1-\la) \mglx{\frac{\norm{x_0 - x_1}}{R}}$ and $x_\la = \la x_0 + (1-\la)x_1.$ Note that \Href{Lemma}{lem:dist_to_pr_sm_set} implies that $P_A([x_0,x_1])\subset \bigcup\limits_{\la\in[0,1]} \BF{r}{x_\la}$.

First, we will show that $ x_1 \notin \BF{r}{x_\la}.$
That is,  we need to verify the following inequality  $r < \norm{x_\la - x_1} = \la \norm{x_1 - x_0},$ which is equivalent to
\[
 8 (1-\la) \frac{R}{\norm{x_0 - x_1}} \mglx{\frac{\norm{x_0 - x_1}}{R}} < 1.
\]
Since $1 - \la \in (0,1),$ this inequality holds whenever $\omx{\frac{\norm{x_0 - x_1}}{R}} \leq 1/8.$
Thus,  $ x_1 \notin \BF{r}{x_\la}.$

Denote the intersection point  of ray $x_\la x_1$ with the boundary of the ball $\BF{r}{x_\la}$ by $v$ and let $\ell$ be one of the two lines  passing through $x_1$ supporting $\BF{r}{x_\la}$.  The tangent point of $\ell$ and  $\BF{r}{x_\la}$ is denoted by $w.$  Note that $y$ is the directional vector of the line supporting $\BF{r}{x_\la}$ at $v.$ Therefore, the lines $\ell$ and $x_{1/2} + \Lin\{y\}$ are not parallel and their intersection point, denoted by $z,$ lies in the same half-plane with the point $w.$  
By similarity, it suffices to set $r^\prime$ equal  to any upper bound on $\norm{x_{1/2} - z}$ that does not depend on $\la$ and $y.$ 

Let us estimate $\norm{x_{1/2} - z}.$ Denote the intersection point of the ray $x_1 x_\la$ and the line $w + \Lin \{y\}$ by $v'.$ By similarity,
we get
\begin{equation}
\label{eq:similarity_waist_of_sauasage}
\norm{x_{1/2} - z} = \frac{1}{2} \norm{x_0 - x_1} \frac{\norm{w - v'}}{\norm{v' - x_1}}.
\end{equation}
Since $y \urcorner (x_1 - x_0),$ we have that $v' \in \BF{r}{x_\la}.$
Hence, we get 
\[
 \norm{w - v'} \leq 2r \quad \text{and} \quad \norm{v' - x_1} \geq \norm{x_\la - x_1} - r = \la \norm{x_1 - x_0} - r.
\] 
Combining these inequalities with inequality \eqref{eq:similarity_waist_of_sauasage}, we get
\[
\norm{x_{1/2} - z} =  \frac{r}{ \la - \frac{r}{\norm{x_0 - x_1}}} =
\frac{8 R (1 - \la) \mglx{\frac{\norm{x_0 - x_1}}{R}}}
{1 - 8  (1 - \la) \omx{\frac{\norm{x_0 - x_1}}{R}}} < 
\frac{8 R  \mglx{\frac{\norm{x_0 - x_1}}{R}}}
{1 - 8   \omx{\frac{\norm{x_0 - x_1}}{R}}} = r^{\prime}.
\]
This completes the proof.
\end{proof}

As an immediate corollary, we get.
\begin{cor}\label{cor:waist_distance_weak_c} 
Let $X$ be a uniformly convex and uniformly smooth  Banach space, let $A \subset X$ be a proximally smooth with constant $R$ set. 
Let  $x_0, x_1 \in A$ with   $\frac{\norm{x_0 - x_1}}{R} < \omxi{1/8},$   fix  $\la \in [0,1]$ and set $x_\la = \la x_0 + (1 - \la)x_1.$ Then there is a point $z_\la \in P_A ([x_0,x_1])$ such that  
\[\norm{z_\la - x_{\la}} \leq 4\la(1-\la) R^\prime (\norm{x_0 -x_1}, R)
\] and
 $(z_\la - x_\la) \urcorner (x_1 - x_0).$ Moreover, 
 the distance between $ \frac{x_0+x_1}{2}$ and any point of the slice projection of the midpoint  
 \( \frac{x_0+x_1}{2} \) of the segment $[x_0, x_1]$ onto $A$
 is at most $R^\prime (\norm{x_0 -x_1}, R).$
\end{cor}
\section{Construction of a curve}
\label{sec:curve_construction}
\subsection{Assumptions on the distance between the endpoints}
Our algorithm of curve construction between two distinct points $x_0$ and $x_1$ of a proximally smooth set works  when points are sufficiently close.  Moreover, we need different bounds to prove the convergence of the algorithm and, for example, to prove the inclusion in \Href{Theorem}{thm:rectifable_curve_inclusion}.
We have decided to collect all the assumptions on the distance between the two starting points.

By definition put
\begin{equation}
\label{eq:mu}
\mu = 
\zetap{\frac{2 R^\prime(\norm{x_0 - x_1}, R)}{\norm{x_0 - x_1}}}
\end{equation}

and recall the definition of $R^\prime(\tau,R)$ (see \eqref{eq:def_Rprime}).

Assumptions on $\frac{\norm{x_0 - x_1}}{R}$:
\begin{enumerate}
\item\label{assump:1} $\frac{\norm{x_0 - x_1}}{R} < \omxi{1/8}.$
\item\label{assump:2}  $\mu < 2.$
\item\label{assump:3} $\frac{\mu^s}{2^{s-1}} < 1.$
\end{enumerate}

In the next lemma, we show that all these assumptions are fulfilled for a sufficiently small $\frac{\norm{x_0 - x_1}}{R}.$
\begin{lem}
\label{lem:small_mu}
Function
$
\zetap{\frac{2 R'(\tau, R)}{\tau}}  
$
is increasing  in $\tau$ on 
$\left[0, R \cdot \omxi{1/8}\right),$ and 
$ \mu \to 1$ as $\frac{\norm{x_0 -x_1}}{R} \to 0.$
\end{lem}
\begin{proof}
By \Href{Proposition}{prop:zetap_mglx_equivalence},
it suffices to show that 
\[
\frac{R^\prime(\norm{x_0 -x_1}, R)}{\norm{x_0 -x_1}} \stackrel{\eqref{eq:def_Rprime}}{=} 
 \frac{8   \omx{\frac{{\norm{x_0 -x_1}}}{R}}}
{1 - 8   \omx{\frac{\norm{x_0 -x_1}}{R}}}
\]
tends to zero as $\frac{\norm{x_0 -x_1}}{R}$ tends to zero,
which is an immediate consequence of the uniformly smoothness of $X$. The monotonicity follows from the monotonicity of 
$\omx{\cdot}.$
\end{proof}

As for the first two of the assumptions, we can bound the corresponding constants using characteristics of a Banach space.
The following is a purely technical result, we formulate it as a separate statement and prove it later in  \Href{Section}{sec:tech_lemmas}.
\begin{claim}\label{claim:nasty}
Set 
\begin{equation}
\label{eq:definition_C2(X)_curve}
\beta_L = \omxi{\frac{(\zeta_X^{+})^{-1}(2)}{8(2 + (\zeta_X^{+})^{-1}(2))}}.
\end{equation}
Then $\beta_L \leq \omxi{\frac{1}{8}} < 2$, and for any positive constants
$\tau$ and $R$ satisfying  $\frac{\tau}{R} < \beta_L,$
inequality
$
\zetap{\frac{2 R'(\tau, R)}{\tau}} < 2
$
holds.
\end{claim}

\subsection{Algorithm for the construction of a curve}
\noindent
 
\begin{algorithm}[H]
\label{algorithm}
\SetAlgoLined 
\KwData{A proximally smooth with constant $R$ set $A \subset X$, two distinct points $x_0,x_1$ in $A$ with $\frac{\norm{x_0 -x_1}}{R} < \beta_L$, where $\beta_L$ is given by \eqref{eq:definition_C2(X)_curve}.}
\KwResult{ A rectifiable curve $f([0,1]),$ where  $f \colon [0,1] \to A $ is a continuous function with 
$f(0) = x_0$ and $f(1) = x_1$}
Set $S_0 = \{0,1\}$ and $S_i = \{\frac{j}{2^i} \mid j \in [2^i]\}\} \cup \{0\}$ for $i \in \N.$
\begin{enumerate}
 \item Define $f$ at points of $S_0$ as follows: $f(0) = x_0$ and $f(1) = x_1.$
 \item  For every $i \in \N,$ we extend the domain of $f$ to the set $S_i \setminus S_{i-1}$ as follows:
 
set the value of $f$ at $\frac{2j-1}{2^i}$ to be a slice-projection of the midpoint  of the segment 
$\left[f\left(\frac{j-1}{2^{i-1}}\right) f\left(\frac{j}{2^{i-1}}\right)\right]$ for all $j \in [2^{i-1}]$ on $A$.
\item Continuously extend $f$ on $[0,1].$ 
 \end{enumerate} 
 \caption{Construction of a curve in a proximally smooth set}
\end{algorithm}
\begin{center}
\begin{figure}[t]
\begin{minipage}[h]{0.49\linewidth}
\center{		\begin{picture}(100,130)
		 \put(-40,0){\includegraphics[scale=0.6]{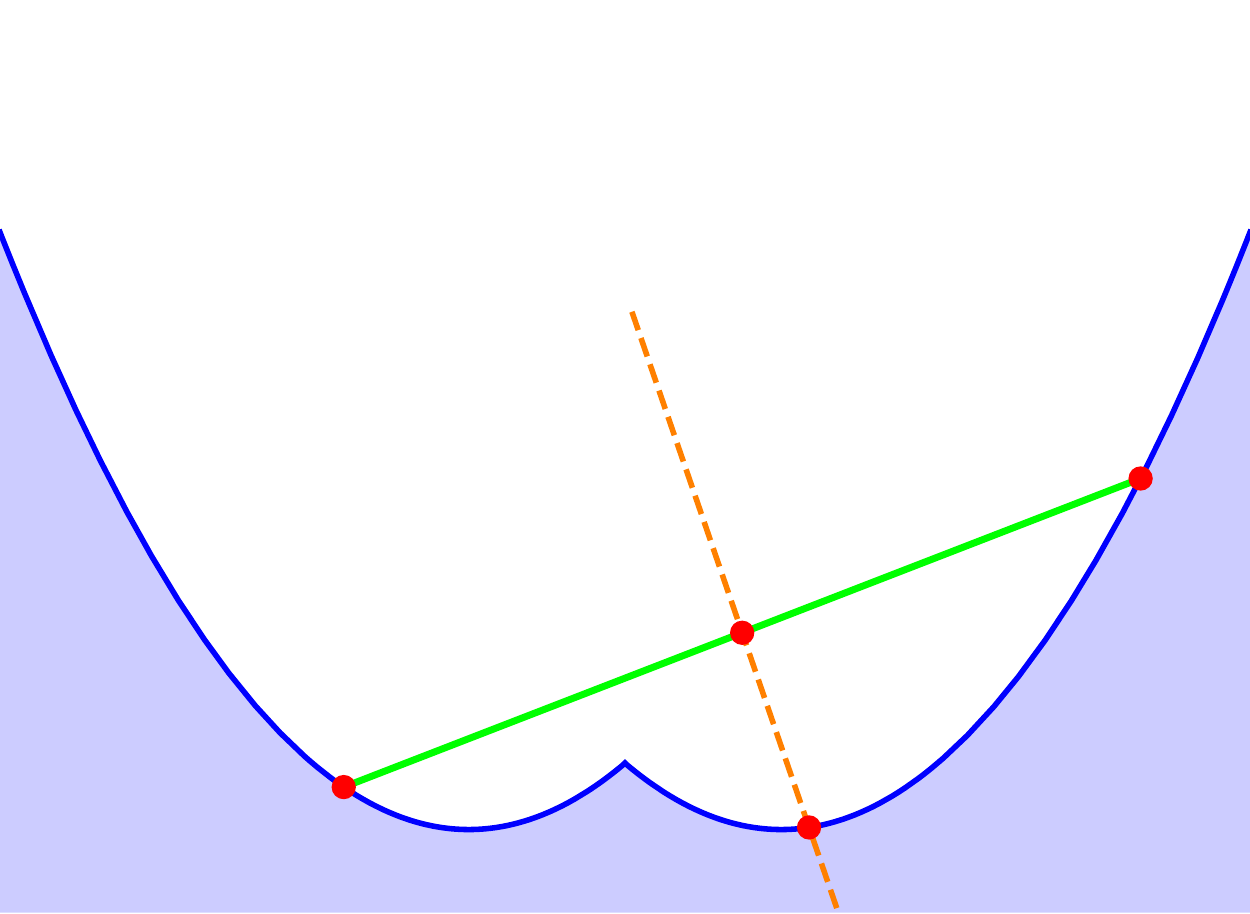}}
		 \put(15 , 30){$x_0$}
		 \put(145 , 80){$x_1$}
		 \put(87 , 58){$\frac{x_0 + x_1}{2}$}
		 \put(90 , 23){$f\!\left(\frac{1}{2} \right)$}
			\end{picture}}
\end{minipage}
\hfill
\begin{minipage}[h]{0.49\linewidth}
		\begin{picture}(100,130)
		 \put(30,0){\includegraphics[scale=0.6]{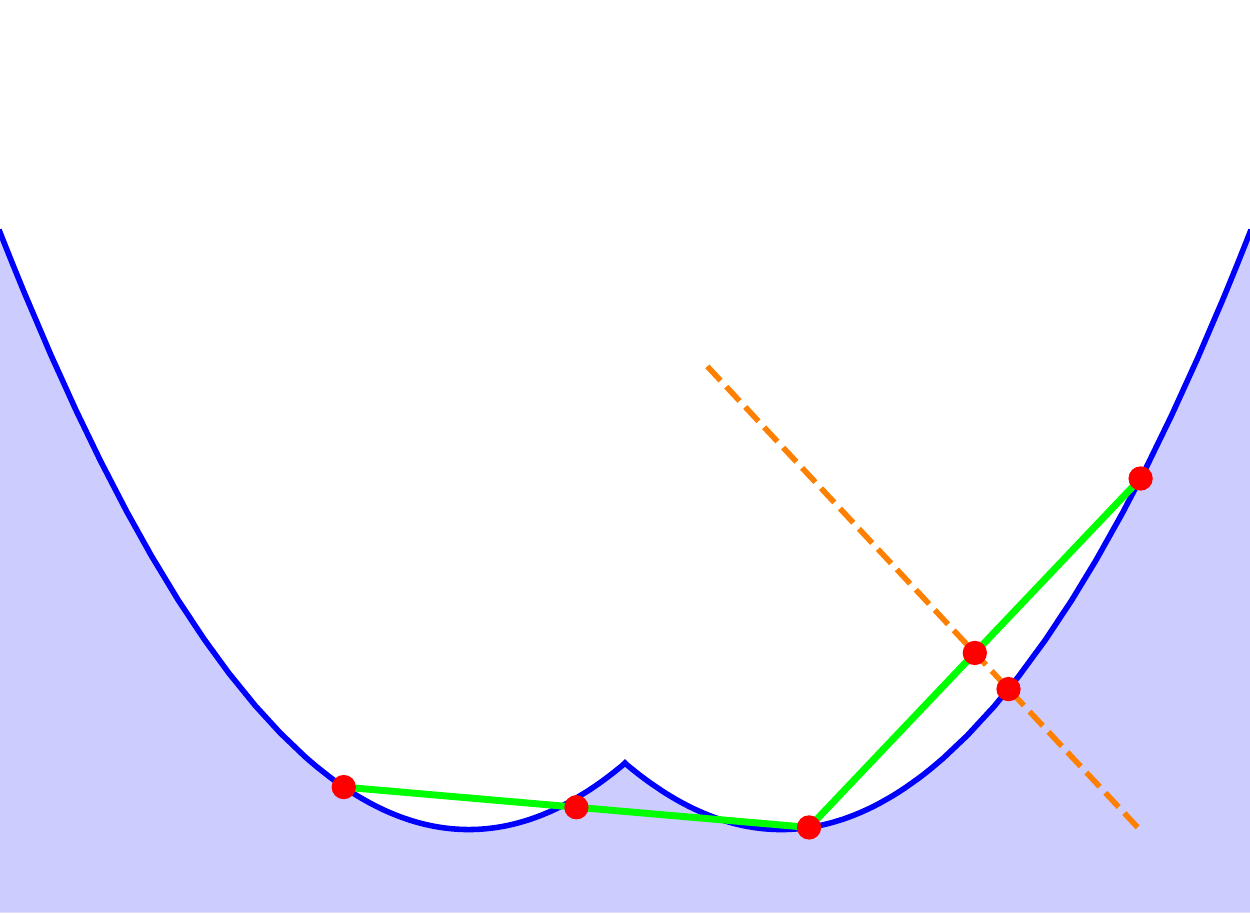}}
		 		 \put(85 , 30){$x_0$}
		 \put(215 , 80){$x_1$}
		 \put(208 , 33){$f\!\left(\frac{3}{4} \right)$}
		 \put(160 , 23){$f\!\left(\frac{1}{2} \right)$}
		 \put(114 , 29){$f\!\left(\frac{1}{4} \right)$}
		 \end{picture}
\end{minipage}
\caption{The first two iterations of the algorithm.}
\end{figure}
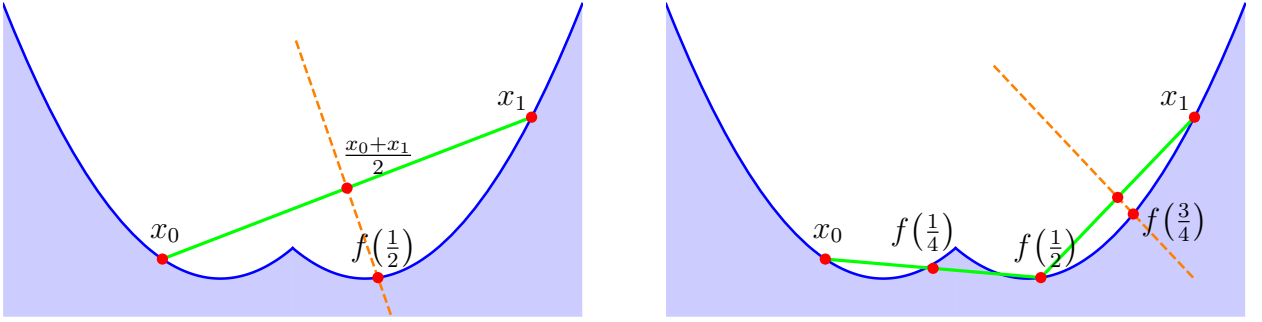
\end{center}
\subsection{Problems needed to be  justified}
 To show the correctness of  \Href{Algorithm}{algorithm},
 one needs to check:
 \begin{enumerate}
 \item For every $i \in \N$ and $j \in [2^{i-1}],$
 there exists a slice-projection of the midpoint of a segment
 $\left[f\left(\frac{j-1}{2^{i-1}}\right) f\left(\frac{j}{2^{i-1}}\right)\right]$ onto the set $A.$
 \item $f$ can be continuously extended from the rational numbers of $[0,1]$
 to the whole segment.
 \item Curve $f([0,1])$ is rectifiable.
 \end{enumerate}

According to \Href{Lemma}{lem:slice_proj_existence} to show the existence of a slice-projection at each step, it suffices to show that  the length  of segment $\left[f\left(\frac{j-1}{2^{i-1}}\right) f\left(\frac{j}{2^{i-1}}\right)\right]$ is less than 
$R \omxi{1/8}.$ We will justify these questions for
 $x_0, x_1$ and $R$ satisfying assumption \ref{assump:2}.

\section{Bound on length}
\label{sec:length_bound}
\Href{Theorem}{thm:rectifable_curve} is an immediate consequence of the following theorem.
\begin{thm}\label{thm:length}
Under the condition of  \Href{Theorem}{thm:rectifable_curve},
additionally let  the modulus of smoothness of $X$ satisfy inequality \eqref{eq:modulus_smoothness_power_type} and  
\[
\frac{\norm{x_0 -x_1}}{R} < 
\beta_L = \omxi{\frac{(\zeta_X^{+})^{-1}(2)}{8(2 + (\zeta_X^{+})^{-1}(2))}}.
\]
Then \Href{Algorithm}{algorithm} returns curve $\gamma$
satisfying
the following inequality:
\[
\length{\gamma} \leq \norm{x_0 - x_1}
\exp\left[\left(\frac{16}{5}\right)^s C_{sm}^{s+1}\frac{1}{1-\left(\frac{\mu}{2}\right)^{s(s-1)}}\left(\frac{\norm{x_0 - x_1}}{2R}\right)^{s(s-1)}\right],
\] 
where $\mu$ is given by \eqref{eq:mu}.
\end{thm}
\begin{proof}
We denote by  $\gamma_i$  the polygonal curve with consecutive vertices of
$\left\{f(t) \mid t \in S_i\right\}$ and by $\Delta_i$  the largest length of a segment of $\gamma_i,$
$i \in \N \cup \{0\}.$

We start with an upper bound on $\Delta_i.$ 
By construction, 
we have that
\begin{equation}
\label{eq:delta_i_trivial}
\Delta_i \leq \frac{\Delta_{i-1}}{2} \zetap{\frac{2 R^\prime(\Delta_{i-1}, R)}{\Delta_{i-1}}}.
\end{equation}

To have a meaningful bound, one needs to guarantee that the argument of $\zetap{\cdot}$ is less than one. 
Using  \Href{Claim}{claim:nasty},  we see that $\Delta_i < \Delta_{i-1}$ starting with $i =1.$ Hence, we have
\begin{lem}
\begin{equation}
\label{eq:Delta_bound_0}
\Delta_i \leq \frac{\mu}{2} \Delta_{i-1}  \leq \left(\frac{\mu}{2}\right)^i \Delta_{0}  \to 0 \quad \text{as} \ i \to \infty.
\end{equation}
\end{lem}

Denote $\psif{\tau} = \zetap{\frac{2 R^\prime(\tau, R)}{\tau}} -1.$
Using \eqref{eq:Delta_bound_0} in \eqref{eq:delta_i_trivial}, we obtain
\begin{equation*}
%\label{eq:delta_i_nontrivial}
\Delta_i \leq \frac{\Delta_{i-1}}{2} \left(1 + \psif{\left(\frac{\mu}{2}\right)^i \Delta_{0} }\right) \leq 
\frac{\Delta_{i-1}}{2} \exp\left[\psif{\left(\frac{\mu}{2}\right)^i \Delta_{0} }\right] .
\end{equation*}
Therefore, we have
\begin{equation*}
%\label{eq:delta_i_nontrivial}
\Delta_i \leq 
\frac{\Delta_{0}}{2^i} \exp\left[\sum_0^i \psif{\left(\frac{\mu}{2}\right)^i \Delta_{0} }\right] .
\end{equation*}
Finally,
\begin{equation}
\label{eq:length_i}
\length{\gamma_i} \leq 2^i \Delta_i \leq \Delta_0 \exp\left[\sum_0^i \psif{\left(\frac{\mu}{2}\right)^i \Delta_{0} }\right] \leq
\Delta_0 \exp\left[\sum_0^\infty \psif{\left(\frac{\mu}{2}\right)^i \Delta_{0} }\right].
\end{equation}

We need to bound the series in the rightmost part in \eqref{eq:length_i}. It is a purely technical task that involves only routine computations.  We formulate the following inequality and prove it later in Section \ref{sec:tech_lemmas}.
\begin{claim}\label{claim:bounder}
Let $0 < \Delta_0 < R \cdot \beta_L.$  Then
\[
\exp\left[\sum_0^\infty \psif{\left(\frac{\mu}{2}\right)^i \Delta_{0} }\right]
\leq
\exp\left[\left(\frac{16}{5}\right)^s C_{sm}^{s+1}\frac{1}{1-\left(\frac{\mu}{2}\right)^{s(s-1)}}\left(\Delta_0\right)^{s(s-1)}\right].
\]
\end{claim}
Thus, passing to the limit in \eqref{eq:length_i} as $i$ tends to infinity and using \Href{Claim}{claim:bounder}, we get 
that the lengths of $\gamma_i$ are uniformly bounded.
Therefore,   function $f$ constructed above is uniformly continuous on the rational numbers of  interval $[0,1].$ By routine, it can be extended to the continuous function on the whole interval with values in $A.$ Thus, $f$ defines a continuous curve in $A$.
Using \eqref{eq:length_i} again, one sees that the first variation of $\gamma = f([0,1])$ is bounded. Thus,
$\gamma$ is rectifiable.
\end{proof}

\section{Inclusion}
\label{sec:inclusion}
In this Section we prove \Href{Theorem}{thm:rectifable_curve_inclusion}. 
The proof consists of several steps. Firstly, we show that the curve returned by \Href{Algorithm}{algorithm} is in a cylinder
of a certain radius around line $x_0 x_1.$ Then we
show that the second part of the curve, that is $f([1/2, 1]),$
belongs to a certain convex cone with apex at $x_0.$
Finally, we prove that all parts of the curve of the form
$f([1/2^{k}, 1/2^{k-1}])$ are in a certain cone with apex at $x_0.$

Again, 
we denote by  $\gamma_i$  the polygonal curve with consecutive vertices of
$\left\{f(t) \mid t \in S_i\right\}$ and by $\Delta_i$  the largest length of a segment of $\gamma_i,$
$i \in \N \cup \{0\}.$
Denote   an intersection point of the hyperplane
$H_{x_0 -x_1}+f(t)$ and the line $x_0 x_1$ by $g(t).$
Define $g_2 \colon [0,1] \to [0, \infty)$ and 
$g_1 \colon  [0,1] \to [0, \infty)$ by
\[
g_2 (t) = \norm{x_0 -g(t)} 
\quad \text{and} \quad
g_1 (t) = \norm{f(t) -g(t)}.
\]

We choose  {$\beta_I$} in such a way that
inequality $\frac{\norm{x_0 - x_1}}{R} < \beta_I$ implies
assumption \ref{assump:3}. By \Href{Lemma}{lem:small_mu}
such a constant exists.

To bound  $g_1$ from above and  $g_2$ from below, we need the following purely technical result, which we prove in the next section.
\begin{claim}\label{claim:sum}
Under the conditions of  \Href{Theorem}{thm:rectifable_curve_inclusion}, additionally let
the modulus of smoothness of $X$ satisfy inequality \eqref{eq:modulus_smoothness_power_type} and let 
$\frac{{\norm{x_0 -x_1}}}{R}$ satisfy assumption \ref{assump:3},  then inequality
\[
\sum_{j=0}^k\frac{R^\prime(\Delta_j, R)}
{2^{k-j}}
<
\frac{24 C_{sm}}{1 - \frac{\mu^{s}}{2^{s-1}}}
\cdot\left(\frac{\Delta_0}{R}\right)^{s-1}
\cdot \frac{\Delta_0}{2^{k}}
\]
holds.
\end{claim}
\begin{lem}\label{lem:final_estimate}
Under the condition of  \Href{Theorem}{thm:rectifable_curve_inclusion}, additionally let
the modulus of smoothness of $X$ satisfy inequality \eqref{eq:modulus_smoothness_power_type} and 
$\frac{{\norm{x_0 -x_1}}}{R}$ satisfy assumption \ref{assump:3}.
For any $t \in [0,1],$ inequality 
\[
g_1(t) \leq \frac{48 C_{sm}}{1 - \frac{\mu^{s}}{2^{s-1}}}
\cdot\left(\frac{\Delta_0}{R}\right)^{s-1} \Delta_0
\]
holds. 
\end{lem}
\begin{proof}
We will greedy estimate $g_1(t),$ $t \in (0,1)$ from above.
We proceed by induction on $k$ and will prove the following
\begin{equation}
\label{eq:g_1_greedy_second_half}
g_1 \! \left(\frac{2j-1}{2^k}\right) \leq 
2\sum_{i=0}^k {R^\prime(\Delta_i, R)}
\quad \forall\ k \in \N,\ j \in [2^{k-1}].
\end{equation}

\Href{Lemma}{lem:weakly_convex_sausage_waist} yields the case 
$k=1.$  Suppose inequality \eqref{eq:g_1_greedy_second_half} holds for $k-1.$ Let us prove it for $k.$
Fix $j \in [2^{k-1}]$ and
denote 
\[a =  \frac{f\!\left(\frac{2j-2}{2^{k}}\right) + f\!\left(\frac{2j}{2^{k}}\right)}{2}
\quad \text{and} \quad 
b = x_0x_1 \cap (H_{x_0 - x_1} + a).
\]
Then, by the triangle inequality,
we get
\[
g_1 \! \left(\frac{2j-1}{2^k}\right) =
\norm{g \! \left(\frac{2j-1}{2^k}\right) - 
f \! \left(\frac{2j-1}{2^k}\right)} \leq
\norm{g \! \left(\frac{2j-1}{2^k}\right) - b} 
+ \norm{b -a} + 
\norm{a - f \! \left(\frac{2j-1}{2^k}\right)}.
\]
Again, by the triangle inequality, 
\[
\norm{b - a }
\leq
\frac{g_1 \!\left(\frac{2j-2}{2^{k}}\right) + g_1 \!\left(\frac{2j}{2^{k}}\right)}{2},
\]
and by construction, we have
\[
\norm{ g \! \left(\frac{2j-1}{2^k}\right) - b} \leq 
 \norm{f \! \left(\frac{2j-1}{2^k}\right) - a} 
 \leq R^\prime( \Delta_k, R).
\] 
Thus, 
\[
g_1 \! \left(\frac{2j-1}{2^k}\right) \leq 
2  \norm{f \! \left(\frac{2j-1}{2^k}\right) - a} + 
\frac{g_1 \!\left(\frac{2j-2}{2^{k}}\right) + g_1 \!\left(\frac{2j}{2^{k}}\right)}{2}
\leq 2 R^\prime( \Delta_k, R) + 2\sum_{i=0}^{k-1} {R^\prime(\Delta_i, R)}.
\]
Inequality \eqref{eq:g_1_greedy_second_half} is proven.

Thus, by \Href{Claim}{claim:sum} and continuity,
we have that
\[
g_1(t) \leq \frac{48 C_{sm}}{1 - \frac{\mu^{s}}{2^{s-1}}}
\cdot\left(\frac{\Delta_0}{R}\right)^{s-1} \Delta_0
\]
for any $t \in [0,1].$

\end{proof}
\Href{Lemma}{lem:final_estimate} says that the curve returned by the algorithm lies in a certain cylinder around line 
$x_0 x_1 .$ To understand curve's behavior near endpoints, we need a more subtle argument.
\begin{lem}\label{lem:final_estimates2}
Under the condition of  \Href{Theorem}{thm:rectifable_curve_inclusion}, additionally let
the modulus of smoothness of $X$ satisfy inequality \eqref{eq:modulus_smoothness_power_type} and 
$\frac{{\norm{x_0 -x_1}}}{R}$ satisfy assumption \ref{assump:3}.
For any $t \in \left[\frac12 , 1\right],$ inequality 
$g_2(t) \ge \frac{\Delta_0}{4}$
holds.
\end{lem}
\begin{proof}%[Proof of \Href{Lemma}{lem:final_estimates2}]
We will greedy estimate $g_2(t),$ $t \in \left(\frac12 , 1\right)$ from below.
We proceed by induction on $k$ and will prove the following
\begin{equation}
\label{eq:g_2_greedy_second_half}
g_2 \! \left(\frac{2j-1}{2^k}\right) \ge
\frac{\Delta_0}{2}-\sum_{i=0}^k {R^\prime(\Delta_i, R)}
\quad \forall\ k \in \N,\ j \in [2^{k-1}], \ j>2^{k-2}.
\end{equation}

The construction of the curve and definition of $g_2(\cdot)$ yield that 
\[
g_2\left(\frac12\right)=\frac{\Delta_0}{2}\ge \frac{\Delta_0}{2}- {R^\prime(\Delta_0, R)}.
\] 

Thus, we have the induction basis. Suppose now that \eqref{eq:g_2_greedy_second_half} holds for some $k-1$. Let us now prove that it holds for $k$. 

Fix $j \in [2^{k-1}]$, $j>2^{k-2}$ and
denote 
\[a =  \frac{f\!\left(\frac{2j-2}{2^{k}}\right) + f\!\left(\frac{2j}{2^{k}}\right)}{2}
\quad \text{and} \quad 
b = x_0x_1 \cap (H_{x_0 -x_1} + a).
\]
Note that $b $ and $g(t)$ lie on the line $x_0 x_1.$
\begin{center}
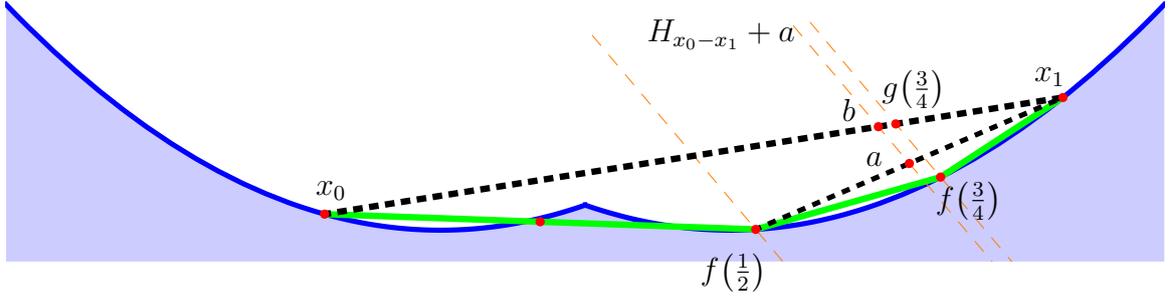
\begin{figure}[t]
		\begin{picture}(100,130)
		 \put(-160,10){\includegraphics[scale=1.2]{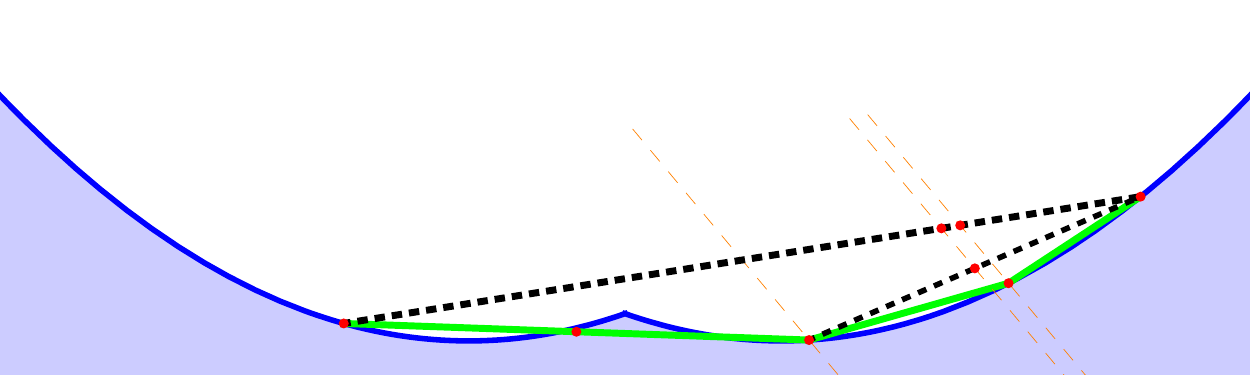}}
		 		 \put(-44 , 35){$x_0$}
		 \put(225 , 78){$x_1$}
		 \put(188 , 30){$f\!\left(\frac{3}{4} \right)$}
		 \put(100 , 3){$f\!\left(\frac{1}{2} \right)$}
		  \put(168 , 72){$g\!\left(\frac{3}{4} \right)$}
		   \put(153 , 63){$b$}
		    \put(162 , 46){$a$}
		     \put(80 , 93){$H_{x_0 - x_1} + a$}
		 \end{picture}
\caption{ Illustration for the proof of \Href{Lemma}{lem:final_estimate}. Here, $k=2$ and $j =2.$ Thus, 
$2a = f\!\left(\frac{1}{2}\right) + f(1) = f\!\left(\frac{1}{2}\right) + x_1$. The dashed lines denotes hyperplanes parallel to $H_{x_0 - x_1}.$ }
\end{figure}
\end{center}
By the triangle inequality, we obtain that
\[
g_2 \! \left(\frac{2j-1}{2^k}\right) =
\norm{x_0 - g \! \left( \frac{2j-1}{2^k}\right)} 
\ge
\norm{x_0 - b} 
-\norm{g \! \left(\frac{2j-1}{2^k}\right) - b} 
.
\]
By the definition of $b$, we have  
\[
\norm{b - x_0 }
=\frac{g_2 \!\left(\frac{2j-2}{2^{k}}\right) + g_2 \!\left(\frac{2j}{2^{k}}\right)}{2}.
\]
As $f(t)-g(t)$ is quasi-perpendicular to $x_0-x_1$, we have that 
\[
\norm{g \! \left(\frac{2j-1}{2^k}\right) - b}\leq \norm{f\! \left(\frac{2j-1}{2^k}\right)-a}.
\] 
Thus, 
\[
g_2\! \left(\frac{2j-1}{2^k}\right) \ge 
\frac{g_2 \!\left(\frac{2j-2}{2^{k}}\right) + g_2 \!\left(\frac{2j}{2^{k}}\right)}{2}
-  \norm{f \! \left(\frac{2j-1}{2^k}\right) - a} 
\ge  \frac{\Delta_0}{2} - R^\prime( \Delta_k, R) -\sum_{i=0}^{k-1} {R^\prime(\Delta_i, R)}.
\]
Inequality \eqref{eq:g_2_greedy_second_half} is proven.

\Href{Claim}{claim:sum}, the assumptions on $\Delta_0$, and the continuity of $g(\cdot)$ imply that $g_2(t)\ge\frac{\Delta_0}{4}$ for all $t\in \left[\frac12,1\right].$
\end{proof}

\begin{cor}
\label{cor:inclusion_second_half}
Under the condition of  \Href{Theorem}{thm:rectifable_curve_inclusion}, additionally let
the modulus of smoothness of $X$ satisfy inequality \eqref{eq:modulus_smoothness_power_type} and 
$\frac{{\norm{x_0 -x_1}}}{R}$ satisfy assumption \ref{assump:3}.
For any $t \in \left[\frac12 , 1\right],$ 
the set 
$\{f(t) \st t \in [\frac{1}{2}, 1] \}$ is a subset of
\[
\mathrm{cone} \left(x_0, \BF{L_1 \Delta_0}{x_1}\cap (H_{x_0 -x_1} + x_1)\right),
\] 
where 
\[
L_1 = \frac{400 C_{sm}}{1 - \frac{\mu^{s}}{2^{s-1}}} t\left(\frac{\Delta_0}{R}\right)^{s-1}.
\]
\end{cor}

Define 
\[
r_i = \frac{1}{1 - \frac{\mu^s}{2^{s-1}}} \cdot \frac{400 C_{sm}}{1 - \frac{\mu^s}{2^{s-1}}}
\cdot\left(\frac{\Delta_i}{R}\right)^{s-1}\Delta_i 
\quad \text{and} \quad 
G_i = \mathrm{cone}\left(x_0, 
\BF{r_i}{f\left(\frac{1}{2^{i-1}}\right)}\right),
\ i \in \N.
\]
\begin{lem}\label{lem:final_estimate2}
Under the condition of  \Href{Theorem}{thm:rectifable_curve_inclusion}, additionally let
the modulus of smoothness of $X$ satisfy inequality \eqref{eq:modulus_smoothness_power_type} and 
$\frac{{\norm{x_0 -x_1}}}{R}$ satisfy assumption \ref{assump:3}. Then
$f([0,1]) \subset G_1.$ 
\end{lem}
\begin{proof}
By inequality \eqref{eq:Delta_bound_0} and assumption \eqref{assump:3}, we have that $\Delta_{i} \leq \Delta_{i-1}.$
Hence, $R^\prime (\Delta_i, R) \leq R^\prime (\Delta_{i-1}, R).$ Thus, Corollary \ref{cor:inclusion_second_half} implies that
$f([1/2^{i}, 1/{2^{i-1}}]) \subset G_i$ for every $i \in \N.$
Thus, to prove the lemma, it suffices to show that $G_i \subset G_1$ for all $i \in \N.$
By construction, the curve $f([0, 1/2^i])$ coincides with the
curve returned by  \Href{Algorithm}{algorithm} applied to points $x_0$ and $f(1/2^{i}).$

Let us prove that $G_{i}\subset G_{i-1}$.
By assumption \eqref{assump:3} and by inequality \eqref{eq:Delta_bound_0},  
\begin{equation}
\label{eq:radii_comparison}
r_i \leq \frac{1}{2} \cdot \frac{\mu^s}{2^{s-1}} r_{i-1}.
\end{equation}

Denote the midpoint of a segment $\left[x_0, f(1/{2^{i-1}})\right]$ by $a_i.$ By the triangle inequality,
\[
G_{i+1} \subset \mathrm{cone} (x_0, \BF{r_i + \norm{f(1/2^i) - a_i}}{a_i}) \quad \text{for all} \ i \in \N.
\]
By similarity,
\[
 \mathrm{cone} (x_0, \BF{r_i + \norm{f(1/2^i) - a_i}}{a_i}) =
\mathrm{cone} (x_0, \BF{2r_i + 2\norm{f(1/2^i) - a_i}}{f(1/2^{i-1})})   \quad \text{for all} \ i \in \N.
\]
Applying \Href{Corollary}{cor:waist_distance_weak_c} with
$x_0 = x_0$ and $x_1 = f(1/2^{i}), $ we get that
\[
\norm{f(1/2^i) - a_i} \leq R^\prime (\norm{f(1/2^{i-1} - x_0)}, R) \leq R^\prime (\Delta_i, R) .
\]
Since $\Delta_i \leq \Delta_0 \leq \beta_I$ and by \Href{Claim}{claim:sum}, 
\[
\norm{f(1/2^i) - a_i} \leq 
 \frac{24 C_{sm}}{1 - \frac{\mu^{s}}{2^{s-1}}} \left(\frac{\Delta_{i-1}}{R}\right)^{s-1} \Delta_{i} 
 \leq
 \left(1 - \frac{\mu^{s}}{2^{s-1}}\right) r_{i-1}.
\]
This and inequality \eqref{eq:radii_comparison} imply that
\[
2r_i + 2\norm{f(1/2^i) - a_i} \leq r_{i-1}. 
\]
Consequently, by the triangle inequality, one has 
\[
G_{i+1} \subset 
\mathrm{cone} (x_0, \BF{r_i + \norm{f(1/2^i) - a_i}}{a_i})
\subset G_i \quad \text{for all} \ i \in \N.
\]
Hence,  we conclude that $G_i \subset G_1$ for all $i \in \N,$
completing the proof of the lemma.
\end{proof}
 
 By symmetry and by \Href{Lemma}{lem:final_estimate2}, we get the following result which implies \Href{Theorem}{thm:rectifable_curve_inclusion}.
\begin{thm}
Under the condition of  \Href{Theorem}{thm:rectifable_curve_inclusion}, additionally let
the modulus of smoothness of $X$ satisfy inequality \eqref{eq:modulus_smoothness_power_type} and $\beta_I$ 
satisfy inequality
\[
\frac{\nu^s}{2^{s-1}} < 1,
\quad
\text{where } \  
\nu = 
\zetap{\frac{2 R^\prime(\beta_I, R)}{\beta_I}}.
\]
Then the curve $\gamma$ returned by \Href{Algorithm}{algorithm}
satisfies inclusion
\begin{equation*}
%\label{eq:thm_inclusion_curve}
 \gamma \subset  
 \conv \left\{
 x_0, \BF{r}{\frac{x_0 + x_1}{2}}, x_1 \right\},
\end{equation*}
 where 
\[
r = \frac{1}{1 - \frac{\mu^s}{2^{s-1}}} \cdot \frac{400 C_{sm}}{1 - \frac{\mu^s}{2^{s-1}}} \norm{x_0 - x_1 } 
\left(\frac{\norm{x_0 - x_1}}{R}\right)^{s-1}. 
\]
\end{thm}

\section{Proofs of   technical results}
\label{sec:tech_lemmas}
\begin{proof}[Proof of  \Href{Claim}{claim:nasty}]
Denote $\zeta^{-1} = \left(\zeta_X^{+}\right)^{-1}(2).$
The definition of $\zetap{\cdot}$ implies that 
\begin{equation}
\label{eq:inequality_zeta}
 1 + \tau \geq \zetap{\tau} \geq 1.
\end{equation} 
Since $\omxi{\cdot}$ is an increasing function, one has  $ \omxi{3/40} <  \omxi{1/8}.$ Thus,
to show that $\beta_L \leq \omxi{1/8}, $ and it suffices to show that $\frac{\zeta^{-1}}{2 + \zeta^{-1}}\leq \frac{3}{5}.$
By \eqref{eq:inequality_zeta} and by monotonicity, we obtain that 
$$\frac{\zeta^{-1}}{2 + \zeta^{-1}} \leq
\left.\frac{\tau+1}{2+\tau+1}\right|_{\tau=2}=\frac{3}{5}.$$
Since $ R^\prime(\tau, R) = \tau \frac{8   \omx{\frac{{\tau}}{R}}}
{1 - 8   \omx{\frac{\tau}{R}}}$ (see \eqref{eq:def_Rprime}), we have 
\[
\zetap{\frac{2 R^\prime(\tau, R)}{\tau}}=\zetap{\frac{16\omx{\frac{\tau}{R}}}{1 - 8   \omx{\frac{\tau}{R}}}}.
\]
Since $\frac{\tau}{R} < \beta_L$ and functions $\zetap{\cdot}$ and $\omx{\cdot}$ are increasing, we obtain that
\[
\zetap{\frac{16\omx{\frac{\tau}{R}}}{1 - 8   \omx{\frac{\tau}{R}}}}<
\zetap{\frac{16\omx{\beta_L}}{1-8\omx{\beta_L}}}=
\zetap{\frac{\frac{16\zeta_X^{-1}(2)}{8(2 + \zeta_X^{-1}(2))}}{1-\frac{\zeta_X^{-1}(2)}{(2 + \zeta_X^{-1}(2))}}}=2
\]
completing the proof of  \Href{Claim}{claim:nasty}.

Let us prove  inequality $\omxi{1/8} < 2.$
By monotonicity of $\omxi{\cdot}$ and by the definition
of $\omx{\cdot},$ we have the following chain
\[
\omxi{1/8} < 2  \ \ \Leftrightarrow \ \ 
1/8 < \omx{2}  \ \ \Leftrightarrow \ \ 
1/4  < \mglx{2}.
\]
The last inequality follows from \eqref{eq:day_nord_mglx}. 
\end{proof}

\begin{proof}[Proof of  \Href{Claim}{claim:bounder}]
By \Href{Proposition}{prop:zetap_mglx_equivalence} and the definition of $R^\prime(\tau, R)$  (see \eqref{eq:def_Rprime}), we get 
\[
\psif{\tau} = \zetap{\frac{2 R^\prime(\tau, R)}{\tau}} -1\leq\mglx{{4R^\prime(\tau, R)}{\tau}}=
\mglx{\frac{32\omx{\frac{\tau}{R}}}{1 - 8\omx{\frac{\tau}{R}}}}.
\]
 \Href{Claim}{claim:nasty} and inequality $\mglx{\tau} \leq C_{sm} \tau^s$ imply that

\[
\mglx{\frac{32\omx{\frac{\tau}{R}}}{1 - 8\omx{\frac{\tau}{R}}}}\leq\mglx{\frac{32\frac{\mglx{\frac{\tau}{R}}}{\frac{\tau}{R}}}{1 -8\omx{C_2}}}\leq\mglx{\frac{32C_{sm}\left(\frac{\tau}{R}\right)^{s-1}}{1-\frac{3}{5}}}\leq \left(\frac{16}{5}\right)^s C_{sm}^{s+1}\left(\frac{\tau}{R}\right)^{s(s-1)}.
\]
Thus,
$$\sum_0^\infty \psif{\left(\frac{\mu}{2}\right)^i \Delta_{0} }\leq \left(\frac{16}{5}\right)^s C_{sm}^{s+1}\Delta_0^{(s-1)s} \sum_0^\infty \left(\frac{\mu}{2}\right)^{i(s(s-1))} =\left(\frac{16}{5}\right)^s C_{sm}^{s+1}\Delta_0^{(s-1)s}\frac{1}{1-\left(\frac{\mu}{2}\right)^{s(s-1)}}.$$

\end{proof}

\begin{proof}[Proof of \Href{Claim}{claim:sum}]
Denote 
\[
S_k = 
\frac{1}{8 R}\sum_{j=0}^k\frac{R^\prime(\Delta_j, R)}
{2^{k-j}}=
\sum_{j=0}^k\frac{  \mglx{\frac{\Delta_j}{R}}}
{2^{k-j}\left(1 - 8   \omx{\frac{\Delta_j}{R}}\right)}.
\]

Taking into account that  $1 - 8   \omx{\frac{\Delta_j}{R}}\ge \frac{2}{5}$, we obtain that

\[
S_k\leq \frac{5}{2}\cdot\sum_{j=0}^k \frac{ 1 }{2^{k-j}} \mglx{\frac{\Delta_j}{R}}.
\]

Hence, considering that ${\Delta_j}\leq\left(\frac{\mu}{2}\right)^j\Delta_0$ and $\mglx{\tau}\leq C_{sm}\tau^s$, we get

\[
S_k \leq 
\frac{5}{2}\cdot\sum_{j=0}^k \frac{ C_{sm} }{2^{k-j}} \left(\left(\frac{\mu}{2}\right)^j \frac{\Delta_0}{R}\right)^s =
\frac{5}{2}\cdot\frac{ C_{sm} }{2^{k}}  
\left(\frac{\Delta_0}{R}\right)^s
\sum_{j=0}^k \frac{\mu^{js}}{2^{js-j}}
 < 
 \frac{5}{2}\cdot\frac{ C_{sm} }{2^{k}}  
\left(\frac{\Delta_0}{R}\right)^s
\sum_{j=0}^{\infty}\left(\frac{\mu^{s}}{2^{s-1}}\right)^{j}.
\]

Assumption \eqref{assump:3} 
($\frac{\mu^{s}}{2^{s-1}}  < 1$) yields that
\[\sum_{j=0}^{\infty}\left(\frac{\mu^{s}}{2^{s-1}}\right)^j
= \frac{1}{ 1 - \frac{\mu^{s}}{2^{s-1}}}
< \infty.
\]
Finally, 
we obtain that
\[
S_k < 
\frac{3 C_{sm}}{1 - \frac{\mu^{s}}{2^{s-1}}}
\cdot\left(\frac{\Delta_0}{R}\right)^s
\cdot \frac{1}{2^{k}},
\]
completing the proof of \Href{Claim}{claim:sum}.
\end{proof}

\bibliographystyle{alpha}
\bibliography{../uvolit}

\end{document}